\documentclass[10pt]{amsart}
\usepackage{amssymb}
\usepackage{bm}
\usepackage{graphicx}
\usepackage[centertags]{amsmath}
\usepackage{amsfonts}
\usepackage{amsthm}
\newtheorem{thm}{Theorem}
\newtheorem{cor}[thm]{Corollary}
\newtheorem{lem}[thm]{Lemma}
\newtheorem{prop}[thm]{Proposition}
\newtheorem{defn}[thm]{Defintion}
\theoremstyle{definition}
\newtheorem{rem}{Remark}
\newtheorem{examp}{Example}
\newtheorem{notation}{Notation}
\newcommand{\rr}{\mathbb{R}}
\newcommand{\nn}{\mathbb{N}}
\newcommand{\qq}{\mathbb{Q}}
\newcommand{\ee}{\varepsilon}
\newcommand{\con}{\smallfrown}
\newcommand{\lbf}{\mathcal{L}_{\mathbf{f}}}
\newcommand{\llf}{\mathcal{L}_{\mathbf{f},f}}
\newcommand{\kk}{\mathcal{K}}
\newcommand{\ii}{\mathcal{I}}
\newcommand{\jj}{\mathcal{J}}
\newcommand{\fin}{\mathrm{Fin}}
\newcommand{\SB}{\mathbf{\Sigma}}
\newcommand{\PB}{\mathbf{\Pi}}

\newcommand{\bt}{\mathbb{N}^{<\mathbb{N}}}
\newcommand{\ct}{2^{<\mathbb{N}}}
\newcommand{\sg}{\sigma}
\newcommand{\bs}{\mathbf{f}=(f_n)_n}
\newcommand{\tr}{\mathrm{Tr}}
\newcommand{\wf}{\mathrm{WF}}
\newcommand{\bx}{\mathcal{B}(X)}

\newcommand{\ld}{\mathcal{L}_{\mathcal{D}'}}

\begin{document}

\title[Codings of separable compacta]{Codings of separable compact subsets
of the first Baire class}
\author{Pandelis Dodos}
\address{National Technical University of Athens, Faculty of Applied Sciences,
Department of Mathematics, Zografou Campus, 157 80, Athens, Greece}
\email{pdodos@math.ntua.gr}
\maketitle

\footnotetext[2]{Research supported by a grant of EPEAEK program "Pythagoras".}


\begin{abstract}
Let $X$ be a Polish space and $\kk$ a separable compact subset of the first Baire
class on $X$. For every sequence $\bs$ dense in $\kk$, the descriptive set-theoretic
properties of the set
\[ \lbf=\{ L\in[\nn]: (f_n)_{n\in L} \text{ is pointwise convergent}\} \]
are analyzed. It is shown that if $\kk$ is not first countable,
then $\lbf$ is $\PB^1_1$-complete. This can also happen even if
$\kk$ is a pre-metric compactum of degree at most two, in the sense
of S. Todor\v{c}evi\'{c}. However, if $\kk$ is of degree exactly two,
then $\lbf$ is always Borel. A deep result of G. Debs implies that
$\lbf$ contains a Borel cofinal set and this gives a
tree-representation of $\kk$. We show that classical ordinal assignments of
Baire-1 functions are actually $\PB^1_1$-ranks on $\kk$. We also provide
an example of a $\SB^1_1$ Ramsey-null subset $A$ of $[\nn]$
for which there does not exist a Borel set $B\supseteq A$ such that the
difference $B\setminus A$ is Ramsey-null.
\end{abstract}


\section{Introduction}

Let $X$ be a Polish space. A Rosenthal compact on $X$ is a
subset of real-valued Baire-1 functions on $X$, compact in
the pointwise topology. Standard examples of such compacta
include the Helly space (the space of all non-decreasing
functions from the unit interval into itself), the split
interval (the lexicographical ordered product of the unit
interval and the two-element ordering) and the ball of the
double dual of a separable Banach space not containing $\ell_1$.
That the later space is indeed a compact subset of the first Baire
class follows from the famous Odell-Rosenthal theorem \cite{OR},
which states that the ball of the double dual of a separable
Banach space with the weak* topology consists only of Baire-1
functions if and only if the space does not contain $\ell_1$.
Actually this result motivated H. P. Rosenthal to initiate the
study of compact subsets of the first Baire class in \cite{Ro}.
He showed that all such compacta are sequentially compact.
J. Bourgain, D. H. Fremlin and M. Talagrand proved that Rosenthal
compacta are Fr\'{e}chet spaces \cite{BFT}. We refer to \cite{AGR},
\cite{Pol} and \cite{Ro2} for thorough introductions to the
theory, as well as, its applications in Analysis.

Separability is the crucial property that divides this class in
two. As S. Todor\v{c}evi\'{c} has pointed out in \cite{To}, while
non-separable Rosenthal compacta can be quite pathological, the
separable ones are all "definable". This is supported by the work
of many researchers, including G. Godefroy \cite{Go}, A. Krawczyk
\cite{Kra}, W. Marciszewski \cite{Ma}, R. Pol \cite{Pol2} and is
highlighted in the remarkable dichotomies and trichotomies of
\cite{To}.

Our starting point of view is how we can code separable compact
subsets of the first Baire class by members of a standard
Borel space. Specifically, by a \textit{code} of a separable
Rosenthal compact $\kk$ on a Polish space $X$, we mean a
standard Borel space $C$ and a surjection $C\ni c\mapsto
f_c\in \kk$ such that for all $a\in\rr$ the relation
\[ (c,x)\in R_a \Leftrightarrow f_c(x)>a \]
is Borel in $C\times X$. In other words, inverse images of
sub-basic open subsets of $\kk$ are Borel in $C$ uniformly in $X$.

There is a natural object one associates to every separable
Rosenthal compact $\kk$ and can serve as a coding of $\kk$.
More precisely, for every dense sequence $\bs$ in $\kk$ one defines
\[ \lbf=\{ L\in[\nn]: (f_n)_{n\in L} \text{ is pointwise convergent} \}. \]
The Bourgain-Fremlin-Talagrand theorem \cite{BFT} implies
that $\lbf$ totally describes the members of $\kk$, in the
sense that for every accumulation point $f$ of $\kk$ there
exists $L\in \lbf$ such that $f$ is the pointwise limit of
the sequence $(f_n)_{n\in L}$. Moreover, for every $f\in\kk$
one also defines
\[ \llf=\{ L\in[\nn]: (f_n)_{n\in L} \text{ is pointwise convergent to } f\}. \]
Both $\lbf$ and $\llf$ have been studied in the literature.
In \cite{Kra}, Krawczyk proved that $\llf$ is Borel if and only
if $f$ is a $G_\delta$ point of $\kk$. The set $\lbf$ (more precisely
the set $\lbf\setminus\llf$) has been also considered by
Todor\v{c}evi\'{c} in \cite{To}, in his solution of characters
of points in separable Rosenthal compacta.

There is an awkward fact concerning $\lbf$, namely that $\lbf$
can be non-Borel. However, a deep result of G. Debs \cite{Debs}
implies that $\lbf$ always contains a Borel cofinal set and
this subset of $\lbf$ can serve as a coding. This leads to the
following tree-representation of separable Rosenthal compacta.
\bigskip

\noindent \textbf{Proposition A.} \textit{Let $\kk$ be a separable
Rosenthal compact. Then there exist a countable tree $T$ and a
sequence $(g_t)_{t\in T}$ in $\kk$ such that the following hold.
\begin{enumerate}
\item[(1)] For every $\sg\in [T]$ the sequence $(g_{\sg|n})_{n\in\nn}$
is pointwise convergent.
\item[(2)] For every $f\in\kk$ there exists $\sg\in [T]$ such that
$f$ is the pointwise limit of the sequence $(g_{\sg|n})_{n\in\nn}$.
\end{enumerate} }
\bigskip

It is natural to ask when the set $\lbf$ is Borel or,
equivalently, when $\lbf$ can serve itself as a coding (it is easy
to see that $\lbf$ and $\llf$ are always $\PB^1_1$). In this
direction, the following is shown.
\bigskip

\noindent \textbf{Theorem B.} \textit{Let $\kk$ be a separable
Rosenthal compact.
\begin{enumerate}
\item[(1)] If $\kk$ is not first countable, then for every dense
sequence $\bs$ in $\kk$ the set $\lbf$ is $\PB^1_1$-complete.
\item[(2)] If $\kk$ is pre-metric of degree exactly two, then for
every dense sequence $\bs$ in $\kk$ the set $\lbf$ is Borel.
\end{enumerate} }
\bigskip

\noindent Part (1) above is based on a result of Krawczyk. In
part (2), $\kk$ is said to be a pre-metric compactum of degree exactly
two if there exist a countable subset $D$ of $X$ and a countable subset
$\mathcal{D}$ of $\kk$ such that at most two functions in $\kk$
coincide on $D$ and moreover for every $f\in\kk\setminus\mathcal{D}$
there exists $g\in\kk$ with $f\neq g$ and such that $g$ coincides
with $f$ on $D$. This is a subclass of the class of pre-metric
compacta of degree at most two, as it is defined by Todor\v{c}evi\'{c}
in \cite{To}. We notice that part (2) of Theorem B cannot be lifted
to all pre-metric compacta of degree at most two, as there are examples
of such compacta for which the set $\lbf$ is $\PB^1_1$-complete.

We proceed now to discuss some applications of the above approach.
It is well-known that to every real-valued Baire-1 function on a
Polish space $X$ one associates several (equivalent) ordinal
rankings measuring the discontinuities of the function. An
extensive study of them is done by A. S. Kechris and A. Louveau in
\cite{KL}. An important example is the separation rank $\alpha$.
We have the following boundedness result concerning this index.
\bigskip

\noindent \textbf{Theorem C.} \textit{Let $X$ be a Polish space
and $\bs$ a sequence of Borel real-valued functions on $X$. Let }
\[ \lbf^1=\{ L\in[\nn]: (f_n)_{n\in L} \text{ \textit{is pointwise
convergent to a Baire-1 function}}\}. \]
\textit{Then for every $C\subseteq \lbf^1$ Borel, we have }
\[ \sup\{ \alpha(f_L): L\in C\} <\omega_1 \]
\textit{where, for every $L\in C$, $f_L$ denotes the pointwise
limit of the sequence $(f_n)_{n\in L}$. }
\bigskip

\noindent The proof of Theorem C actually is based on the fact that the
separation rank is a parameterized $\PB^1_1$-rank. Theorem C, combined
with the result of Debs, gives a proof of the boundedness result of
\cite{ADK}. Historically the first result of this form is due to J. Bourgain
\cite{Bou}. We should point out that in order to give a descriptive
set-theoretic proof of Bourgain's result one does not need to invoke
Debs' theorem.

Theorem C can also be used to provide natural counterexamples
to the following approximation question in Ramsey theory. Namely,
given a $\SB^1_1$ subset $A$ of $[\nn]$ can we always find a
Borel set $B\supseteq A$ such that the difference $B\setminus A$
is Ramsey-null? A. W. Miller had also asked whether there exists
an analytic set which is not equal to Borel modulo Ramsey-null
(see \cite{Mi}, Problem $1.6^*$). We show the following.
\bigskip

\noindent \textbf{Proposition D.} \textit{There exists a $\SB^1_1$
Ramsey-null subset $A$ of $[\nn]$ for which there does not exist a
Borel set $B\supseteq A$ such that the difference $B\setminus A$
is Ramsey-null.}
\bigskip

\noindent \textbf{Acknowledgments.} Part of this work was done when
I visited the Department of Mathematics at Caltech. I would like to
thank the department for the hospitality and the financial support.
I would also like to thank the anonymous referee for his thorough
report which substantially improved the presentation of the paper.


\section{Preliminaries}

For any Polish space $X$, by $K(X)$ we denote the hyperspace
of all compact subsets of $X$, equipped with the Vietoris
topology. By $\mathcal{B}_1(X)$ (respectively $\mathcal{B}(X)$)
we denote the space of all real-valued Baire-1 (respectively Borel)
functions on $X$. By $\nn=\{0,1,2,...\}$ we denote the natural
numbers, while by $[\nn]$ the set of all infinite subsets of
$\nn$ (which is clearly a Polish subspace of $2^\nn$). For every
$L\in[\nn]$, by $[L]$ we denote the set of all infinite subsets
of $L$. For every function $f:X\to\rr$ and every $a\in\rr$ we set
$[f>a]=\{x:f(x)>a\}$. The set $[f<a]$ has the obvious meaning.

Our descriptive set-theoretic notation and terminology follows
\cite{Kechris}. So $\SB^1_1$ stands for the analytic sets, while
$\PB^1_1$ for the co-analytic. A set is said to be $\PB^1_1$-true
if it is co-analytic non-Borel. If $X, Y$ are Polish spaces,
$A\subseteq X$ and $B\subseteq Y$, we say that $A$ is Wadge
(Borel) reducible to $B$ if there exists a continuous (Borel)
map $f:X\to Y$ such that $f^{-1}(B)=A$. A set $A$ is said to
be $\PB^1_1$-complete if it is $\PB^1_1$ and any other co-analytic
set is Borel reducible to $A$. Clearly any $\PB^1_1$-complete set
is $\PB^1_1$-true. The converse is also true under large cardinal
hypotheses (see \cite{MK} or \cite{Mo}). If $A$ is $\PB^1_1$, then
a map $\phi:A\to\omega_1$ is said to be a $\PB^1_1$-rank on $A$ if
there are relations $\leq_\Sigma$, $\leq_\Pi$ in $\SB^1_1$ and
$\PB^1_1$ respectively, such that for any $y\in A$
\[ \phi(x)\leq \phi(y) \Leftrightarrow x\leq_\Sigma y
\Leftrightarrow x \leq_\Pi y .\]
Notice that if $A$ is Borel reducible to $B$ via a Borel map $f$
and $\phi$ is a $\PB^1_1$-rank on $B$, then the map $\psi:A\to\omega_1$
defined by $\psi(x)=\phi(f(x))$ is a $\PB^1_1$-rank on $A$.
\bigskip

\noindent \textbf{Trees.} If $A$ is a non-empty set, by $A^{<\nn}$
we denote the set of all finite sequences of $A$. We view $A^{<\nn}$
as a tree equipped with the (strict) partial order $\sqsubset$ of
extension. If $s\in A^{<\nn}$, then the length $|s|$ of $s$ is defined
to be the cardinality of the set $\{t: t\sqsubset s\}$. If $s,t\in A^{<\nn}$,
then by $s^\con t$ we denote their concatenation. If $A=\nn$ and $L\in[\nn]$,
then by $[L]^{<\nn}$ we denote the increasing finite sequences in $L$.
For every $x\in A^\nn$ and every $n\geq 1$ we let $x|n=\big(x(0),...,x(n-1)
\big)\in A^{<\nn}$ while $x|0=(\varnothing)$. A tree $T$ on $A$ is a
downwards closed subset of $A^{<\nn}$. The set of all trees on $A$
is denoted by $\tr(A)$. Hence
\[ T\in\tr(A) \Leftrightarrow \forall s,t\in A^{<\nn} \ (t\sqsubset s \
\wedge \ s\in T\Rightarrow t\in T). \]
For a tree $T$ on $A$, the body $[T]$ of $T$ is defined to be the set
$\{ x\in A^\nn: x|n\in T \text{ for all } n\in\nn\}$. A tree $T$ is
called pruned if for every $s\in T$ there exists $t\in T$ with
$s\sqsubset t$. It is called well-founded if for every $x\in A^\nn$
there exists $n$ such that $x|n\notin T$, equivalently if $[T]=\varnothing$.
The set of well-founded trees on $A$ is denoted by $\wf(A)$. If $T$ is
a well-founded tree we let $T'=\{ t: \exists s\in T \text{ with }
t\sqsubset s\}$. By transfinite recursion, one defines the iterated
derivatives $T^{(\xi)}$ of $T$. The order $o(T)$ of $T$ is defined
to be the least ordinal $\xi$ such that $T^{(\xi)}=\varnothing$. If
$S, T$ are well-founded trees, then a map $\phi:S\to T$ is called monotone
if $s_1\sqsubset s_2$ in $S$ implies that $\phi(s_1)\sqsubset \phi(s_2)$
in $T$. Notice that in this case $o(S)\leq o(T)$. If $A, B$ are sets,
then we identify every tree $T$ on $A\times B$ with the set of all
pairs $(s,t)\in A^{<\nn}\times B^{<\nn}$ such that $|s|=|t|=k$ and
$\big( (s(0),t(0)),....,(s(k-1),t(k-1))\big)\in T$. If $A=\nn$, then
we shall simply denote by $\tr$ and $\wf$ the sets of all trees and
well-founded trees on $\nn$ respectively. The set $\wf$ is
$\PB^1_1$-complete and the map $T\to o(T)$ is a $\PB^1_1$-rank on $\wf$.
The same also holds for $\wf(A)$ for every countable set $A$.
\bigskip

\noindent \textbf{The separation rank.} Let $X$ be a Polish space.
Given $A, B\subseteq X$ one associates with them a derivative on
closed sets, by $F'_{A,B}=\overline{F\cap A}\cap \overline{F\cap B}$.
By transfinite recursion, we define the iterated derivatives
$F^{(\xi)}_{A,B}$ of $F$ and we set $\alpha(F,A,B)$ to be the least
ordinal $\xi$ with $F^{(\xi)}_{A,B}=\varnothing$ if such
an ordinal exists, otherwise we set $\alpha(F,A,B)=\omega_1$.
Now let $f:X\to\rr$ be a function. For each pair $a,b\in\rr$ with
$a<b$ let $A=[f<a]$ and $B=[f>b]$. For every $F\subseteq X$ closed let
$F^{(\xi)}_{f,a,b}=F^{(\xi)}_{A,B}$ and $\alpha(f,F,a,b)=\alpha(F,A,B)$.
Let also $\alpha(f,a,b)=\alpha(f,X,a,b)$. The separation rank of $f$ is
defined by
\[ \alpha(f)=\sup\{ \alpha(f,a,b): a,b\in\qq, a<b \}. \]
The basic fact is the following (see \cite{KL}).
\begin{prop}
\label{kl} A function $f$ is Baire-1 if and only if $\alpha(f)<\omega_1$.
\end{prop}


\section{Codings of separable Rosenthal compacta}

Let $X$ be a Polish space and $\bs$ a sequence of Borel real-valued
functions on $X$. Assume that the closure $\kk$ of $\{f_n\}_n$ in
$\rr^X$ is a compact subset of $\bx$. Let us consider the set
\[ \lbf=\{ L\in[\nn]: (f_n)_{n\in L} \text{ is pointwise convergent} \}. \]
For every $L\in\lbf$, by $f_L$ we shall denote the pointwise limit
of the sequence $(f_n)_{n\in L}$. Notice that $\lbf$ is $\PB^1_1$.
As the pointwise topology is not effected by the topology on $X$,
we may (and we will) assume that each $f_n$ is continuous (and so
$\kk$ is a separable Rosenthal compact). By a result of H. P.
Rosenthal \cite{Ro}, we get that $\lbf$ is cofinal. That is,
for every $M\in[\nn]$ there exists $L\in[M]$ such that $L\in\lbf$.
Also the celebrated Bourgain-Fremlin-Talagrand theorem \cite{BFT}
implies that $\lbf$ totally describes $\kk$. However, most important
for our purposes is the fact that $\lbf$ contains a Borel cofinal set.
This is a consequence of the following theorem of G. Debs \cite{Debs}
(which itself is the classical interpretation of the effective version
of the Bourgain-Fremlin-Talagrand theorem, proved by G. Debs in \cite{Debs}).
\begin{thm}
\label{d} Let $Y, X$ be Polish spaces and $(g_n)_n$ be a sequence of
Borel functions on $Y\times X$ such that for every $y\in Y$ the
sequence $\big(g_n(y,\cdot)\big)_n$ is a sequence of continuous
functions relatively compact in $\bx$. Then there exists a Borel
map $\sg:Y\to [\nn]$ such that for any $y\in Y$, the sequence
$\big(g_n(y,\cdot)\big)_{n\in \sg(y)}$ is pointwise convergent.
\end{thm}
Let us show how Theorem \ref{d} implies the existence of a Borel
cofinal subset of $\lbf$. Given $L, M\in[\nn]$ with $L=\{l_0<l_1<...\}$
and $M=\{m_0<m_1<...\}$ their increasing enumerations, let
$L*M=\{ l_{m_0}<l_{m_1}<...\}$. Clearly $L*M\in [L]$ and moreover
the function $(L,M)\mapsto L*M$ is continuous. Let $(f_n)_n$ be as
in the beginning of the section and let $Y=[\nn]$. For every $n\in\nn$
define $g_n:[\nn]\times X\to \rr$ by
\[ g_n(L,x)= f_{l_n}(x) \]
where $l_n$ in the $n^{th}$ element of the increasing enumeration
of $L$. The sequence $(g_n)_n$ satisfies all the hypotheses of Theorem
\ref{d}. Let $\sg:[\nn]\to[\nn]$ be the Borel function such that
for every $L\in[\nn]$ the sequence
\[ \big( g_n(L,\cdot)\big)_{n\in\sg(L)}=(f_n)_{n\in L*\sg(L)} \]
is pointwise convergent. The function $L\to L*\sg(L)$ is Borel and
so the set
\[ A=\{ L*\sg(L): L\in[\nn]\} \]
is an analytic cofinal subset of $\lbf$. By separation we get that there
exists a Borel cofinal subset of $\lbf$. The cofinality of this set in
conjunction with the Bourgain-Fremlin-Talagrand theorem give us the
following corollary.
\begin{cor}
\label{c} Let $X$ be a Polish space and $(f_n)_n$ a sequence of Borel
functions on $X$ which is relatively compact in $\mathcal{B}(X)$.
Then there exists a Borel set $C\subseteq [\nn]$ such that for every
$c\in C$ the sequence $(f_n)_{n\in c}$ is pointwise convergent and for
every accumulation point $f$ of $(f_n)_n$ there exists $c\in C$
with $f=\lim_{n\in c} f_n$.
\end{cor}
In the sequel we will say that the set $C$ obtained by Corollary
\ref{c} is a \textit{code} of $(f_n)_n$. If $\kk$ is a separable
Rosenthal compact and $(f_n)_n$ is a dense sequence in $\kk$,
then we will say that $C$ is the code of $\kk$. Notice that the
codes depend on the dense sequence. If $c\in C$, then by $f_c$
we shall denote the function coded by $c$. That is $f_c$ is the
pointwise limit of the sequence $(f_n)_{n\in c}$.

The following lemma  captures the basic definability properties
of the set of codes. Its easy proof is left to the reader.
\begin{lem}
\label{cl1} Let $X$ and $(f_n)_n$ be as in Corollary \ref{c} and
let $C$ be a code of $(f_n)_n$. Then for every $a\in \rr$ the
following relations
\begin{enumerate}
\item[(i)] $(c,x)\in R_a \Leftrightarrow f_c(x)>a$,
\item[(ii)] $(c,x)\in R'_a \Leftrightarrow f_c(x)\geq a$,
\item[(iii)] $(c_1,c_2,x)\in D_a \Leftrightarrow |f_{c_1}(x)-f_{c_2}(x)|>a$
\end{enumerate}
are Borel.
\end{lem}
The existence of codings of separable Rosenthal compacta gives us
the following tree-representation of them.
\begin{prop}
\label{ct1} Let $\kk$ be a separable Rosenthal compact. Then there
exist a countable tree $T$ and a sequence $(g_t)_{t\in T}$ in $\kk$
such that the following hold.
\begin{enumerate}
\item[(1)] For every $\sg\in [T]$ the sequence $(g_{\sg|n})_{n\in\nn}$
is pointwise convergent.
\item[(2)] For every $f\in\kk$ there exists $\sg\in [T]$ such that $f$
is the pointwise limit of the sequence $(g_{\sg|n})_{n\in\nn}$.
\end{enumerate}
\end{prop}
\begin{proof}
Let $(f_n)_n$ be a dense sequence in $\kk$. We may assume that for
every $n\in\nn$ the set $\{m: f_m=f_n\}$ is infinite. This extra
condition guarantees that for every $f\in\kk$ there exists $L\in\lbf$
such that $f=f_L$. Let $C$ be the codes of $(f_n)_n$. Now we shall
use a common unfolding trick. As $C$ is Borel in $2^\nn$ there exists
$F\subseteq 2^\nn\times \nn^\nn$ closed such that $C=\mathrm{proj}_{2^\nn}F$.
Let $T$ be the unique (downwards closed) pruned tree on $2\times \nn$
such that $F=[T]$. This will be the desired tree. It remains to define
the sequence $(g_t)_{t\in T}$. Set $g_{(\varnothing,\varnothing)}=f_0$.
Let $t=(s,w)\in T$ and $k\geq 1$ with $s\in \ct$, $w\in\bt$ and
$|s|=|w|=k$.  Define $n_t\in\nn$ to be $n_t=\max\{ n<k: s(n)=1\}$,
if the set $\{n<k:s(n)=1\}$ is non-empty, and $n_t=0$ otherwise.
Finally set $g_t=f_{n_t}$. It is easy to check that for every
$\sg\in[T]$ the sequence $(g_{\sg|n})_n$ is pointwise convergent,
and so (1) is satisfied. That (2) is also satisfied follows from
the fact that for every $f\in\kk$ there exists $L\in \lbf$ with
$f=f_L$ and the fact that $C$ is cofinal.
\end{proof}
\begin{rem}
(1) We should point out that Corollary \ref{c}, combined with
J. H. Silver's theorem (see \cite{MK} or \cite{S2}) on the number
of equivalence classes of co-analytic equivalence relations, gives an
answer to the cardinality problem of separable Rosenthal compacta,
a well-known fact that can also be derived by the results of \cite{To}
(see also \cite{ADK}, Remark 3). Indeed, let $\kk$ be one and let
$C$ be the set of codes of $\kk$. Define the following equivalence
relation on $C$, by
\[ c_1 \sim c_2 \Leftrightarrow f_{c_1}=f_{c_2} \Leftrightarrow
\forall x \ f_{c_1}(x)=f_{c_2}(x). \]
Then $\sim$ is a $\PB^1_1$ equivalence relation. Hence, by Silver's
dichotomy, either the equivalence classes are countable or perfectly
many. The first case implies that $|\kk|=\aleph_0$, while the second
one that $|\kk|=2^{\aleph_0}$.\\
(2) Although the set $C$ of codes of a separable Rosenthal compact
$\kk$ is considered to be a Borel set which describes $\kk$ efficiently,
when it is considered as a subset of $[\nn]$ it can be chosen to
have rich structural properties. In particular, it can be chosen to be
hereditary (i.e. if $c\in C$ and $c'\in [c]$, then $c'\in C$) and
invariant under finite changes. To see this, start with a code $C_1$
of $\kk$, i.e. a Borel cofinal subset of $\lbf$. Let
\begin{eqnarray*}
\Phi =\big\{ (F,G) & : & (F\subseteq \lbf) \ \wedge \ (G\cap C_1=\varnothing)
\ \wedge \\
& & [\forall L,M \ (L\in F \ \wedge \ M\subseteq L\Rightarrow M\notin G)]
\ \wedge \\
& & [\forall L,M,s \ (L\in F \ \wedge \ (L\bigtriangleup M=s)\Rightarrow
M\notin G)]\big\}.
\end{eqnarray*}
Let also $A=\{ N : \exists L\in C_1 \ \exists s\in [\nn]^{<\nn} \
\exists M\in[L] \text{ with } N\bigtriangleup M=s \}$. Then $A$ is
$\SB^1_1$ and clearly $\Phi(A,\sim A)$. As $\Phi$ is $\PB^1_1$
on $\SB^1_1$, hereditary and continuous upward in the second variable,
by the dual form of the second reflection theorem (see \cite{Kechris},
Theorem 35.16), there exists $C\supseteq A$ Borel with $\Phi(C,\sim C)$.
Clearly $C$ is as desired.\\
(3) We notice that the idea of coding subsets of function spaces
by converging sequences appears also in \cite{Be}, where a
representation result of $\SB^1_2$ subsets of $C([0,1])$
is proved.
\end{rem}


\section{A boundedness result}

\subsection{Determining $\alpha(f)$ by compact sets}

Let $X$ be a Polish space and $f:X\to\rr$ a Baire-1 function.
The aim of this subsection is to show that the value $\alpha(f)$
is completely determined by the derivatives taken over compact
subsets of $X$ (notice that this is trivial if $X$ is compact
metrizable). Specifically we have the following.
\begin{prop}
\label{p2} Let $X$ be a Polish space, $f:X\to\rr$ Baire-1 and
$a<b$ reals. Then
$\alpha(f,a,b)=\sup\{ \alpha(f,K,a,b): K\subseteq X \text{ compact} \}$.
\end{prop}
The proof of Proposition \ref{p2} is an immediate consequence of
the following lemmas. In what follows, all balls in $X$ are taken
with respect to some compatible complete metric $\rho$ of $X$.
\begin{lem}
\label{al1} Let $X$, $f$ and $a<b$ be as in Proposition \ref{p2}.
Let also $F\subseteq X$ closed, $x\in X$ and $\xi<\omega_1$ be
such that $x\in F^{(\xi)}_{f,a,b}$. Then for every $\ee>0$,
if we let $C=F\cap \overline{B(x,\ee)}$, we have $x\in C^{(\xi)}_{f,a,b}$.
\end{lem}
\begin{proof}
Fix $F$ and $\ee$ as above. For notational simplicity let $U=B(x,\ee)$
and $C=F\cap \overline{B(x,\ee)}$. By induction we shall show that
\[ F^{(\xi)}\cap U \subseteq C^{(\xi)} \]
where $F^{(\xi)}=F^{(\xi)}_{f,a,b}$ and similarly for $C$. This clearly
implies the lemma. For $\xi=0$ is straightforward. Suppose that the
lemma is true for every $\xi<\zeta$. Assume that $\zeta=\xi+1$ is a
successor ordinal. Let $y\in F^{(\xi+1)}\cap U$. As $U$ is open, we have
\[ y \in \overline{F^{(\xi)}\cap U\cap [f<a]} \cap \overline{F^{(\xi)}\cap
U\cap [f>b]}. \]
By the inductive assumption we get that
\[ y\in \overline{ C^{(\xi)}\cap [f<a]} \cap \overline{ C^{(\xi)}\cap
[f>b]}=C^{(\xi+1)} \]
which proves the case of successor ordinals. If $\zeta$ is limit, then
\[ F^{(\zeta)}\cap U= \bigcap_{\xi<\zeta} F^{(\xi)}\cap U \subseteq
\bigcap_{\xi<\zeta} C^{(\xi)}=C^{(\zeta)} \]
and the lemma is proved.
\end{proof}
\begin{lem}
\label{al2} Let $X$, $f$ and $a<b$ be as in Proposition \ref{p2}.
Let also $F\subseteq X$ closed, $x\in X$ and $\xi<\omega_1$ be such
that $x\in F^{(\xi)}_{f,a,b}$. Then there exists $K\subseteq F$
countable compact such that $x\in K^{(\xi)}_{f,a,b}$.
\end{lem}
\begin{proof}
Again for notational simplicity for every $C\subseteq X$ closed
and every $\xi<\omega_1$ we let $C^{(\xi)}=C^{(\xi)}_{f,a,b}$.
The proof is by induction on countable ordinals, as before. For $\xi=0$
the lemma is obviously true. Suppose that the lemma has been proved
for every $\xi<\zeta$. Let $F\subseteq X$ closed and $x\in F^{(\zeta)}$.
Notice that one of the following alternatives must occur.
\begin{enumerate}
\item[(A1)] $f(x)<a$ and there exists a sequence $(y_n)_n$ such that
$y_n\neq y_m$ for $n\neq m$, $f(y_n)>b$, $y_n\in F^{(\xi_n)}$ and $y_n\to x$;
\item[(A2)] $f(x)>b$ and there exists a sequence $(z_n)_n$ such that
$z_n\neq z_m$ for $n\neq m$, $f(z_n)<a$, $z_n\in F^{(\xi_n)}$ and $z_n\to x$;
\item[(A3)] there exist two distinct sequences $(y_n)_n$ and $(z_n)_n$
such that $y_n\neq y_m$ and $z_n\neq z_m$ for $n\neq m$, $f(y_n)<a$,
$f(z_n)>b$,  $y_n, z_n\in F^{(\xi_n)}$ and $y_n\to x$, $z_n\to x$,
\end{enumerate}
where above the sequence $(\xi_n)_n$ of countable ordinals is as follows.
\begin{enumerate}
\item[(C1)] If $\zeta=\xi+1$, then $\xi_n=\xi$ for every $n$.
\item[(C2)] If $\zeta$ is limit, then $(\xi_n)_n$ is an increasing
sequence of successor ordinals with $\xi_n\nearrow \zeta$.
\end{enumerate}
We shall treat the alternative (A1) (the other ones are similar).
Let $(r_n)_n$ be a sequence of positive reals such that
$\overline{B(y_n,r_n)}\cap\overline{B(y_m,r_m)}=\varnothing$ if $n\neq m$
and $x\notin \overline{B(y_n,r_n)}$ for every $n$. Let
$C_n=F\cap\overline{B(y_n,r_n)}$. By Lemma \ref{al1}, we get that
$y_n\in C^{(\xi_n)}_n$. By the inductive assumption, there exists
$K_n\subseteq C_n \subseteq F_n$ countable compact such that
$y_n\in K^{(\xi_n)}_n$. Finally let $K=\{x\}\cup (\bigcup_n K_n)$.
Then $K$ is countable compact and it is easy to see that
$x\in K^{(\zeta)}$.
\end{proof}
\begin{rem}
Notice that the proof of Lemma \ref{al2} actually shows that
\[ \alpha(f,a,b)=\sup\{ \alpha(f,K,a,b): K\subseteq X
\text{ countable compact} \}. \]
Moreover observe that if $\alpha(f,a,b)$ is a successor ordinal,
then the above supremum is attainted.
\end{rem}

\subsection{The main result}

This subsection is devoted to the proof of the following result.
\begin{thm}
\label{t4} Let $X$ be a Polish space and $\bs$ a sequence of Borel
real-valued functions on $X$. Let
\[ \lbf^1=\{ L\in[\nn]: (f_n)_{n\in L} \text{ is pointwise convergent
to a Baire-1 function}\}. \]
Then for every $C\subseteq \lbf^1$ Borel, we have
\[ \sup\{ \alpha(f_L): L\in C\} <\omega_1 \]
where, for every $L\in C$, $f_L$ denotes the pointwise
limit of the sequence $(f_n)_{n\in L}$.
\end{thm}
For the proof of Theorem \ref{t4} we will need the following theorem,
which gives us a way of defining parameterized $\PB^1_1$-ranks
(see \cite{Kechris}, page 275).
\begin{thm}
\label{tp} Let $Y$ be a standard Borel space, $X$ a Polish space
and $\mathbb{D}:Y\times K(X)\to K(X)$ be a Borel map such that
for every $y\in Y$, $\mathbb{D}_y$ is a derivative on $K(X)$.
Then the set
\[ \Omega_{\mathbb{D}}=\{ (y,K): \mathbb{D}_y^{(\infty)}(K)=\varnothing \} \]
is $\PB^1_1$ and the map $(y,K)\to |K|_{\mathbb{D}_y}$ is a
$\PB^1_1$-rank on $\Omega_{\mathbb{D}}$.
\end{thm}
We continue with the proof of Theorem \ref{t4}.
\begin{proof}[Proof of Theorem \ref{t4}]
Let $C\subseteq \lbf^1$ Borel arbitrary. Fix $a,b\in\rr$ with $a<b$.
Define $\mathbb{D}:C\times K(X)\to K(X)$ by
\[ \mathbb{D}(L,K)=\overline{K\cap [f_L<a]}\cap \overline{K\cap [f_L>b]} \]
where $f_L$ is the pointwise limit of the sequence $(f_n)_{n\in L}$.
It is clear that for every $L\in C$ the map $K\to \mathbb{D}(L,K)$
is a derivative on $K(X)$ and that $\alpha(f_L,K,a,b)=|K|_{\mathbb{D}_L}$.
We will show that $\mathbb{D}$ is Borel. Define
$A,B \in C\times K(X)\times X$ by
\[ (L,K,x) \in A \Leftrightarrow (x\in K) \ \wedge \ (f_L(x)<a) \]
and
\[ (L,K,x) \in B \Leftrightarrow (x\in K) \ \wedge \ (f_L(x)>b). \]
It is easy to check that both $A$ and $B$ are Borel. Also let
$\tilde{A}, \tilde{B} \subseteq C\times K(X)\times X$ be defined by
\[ (L,K,x)\in \tilde{A} \Leftrightarrow x\in \overline{A_{(L,K)}} \]
and
\[ (L,K,x)\in \tilde{B} \Leftrightarrow x\in \overline{B_{(L,K)}}, \]
where $A_{(L,K)}=\{ x: (L,K,x)\in A\}$ is the section of $A$
(and similarly for $B$). Notice that for every $(L,K)\in C\times K(X)$
we have $\mathbb{D}(L,K)=\tilde{A}_{(L,K)}\cap \tilde{B}_{(L,K)}$.
As $\tilde{A}_{(L,K)}$ and $\tilde{B}_{(L,K)}$ are compact (being
subsets of $K$), by Theorem 28.8 in \cite{Kechris}, it is enough
to show that the sets $\tilde{A}$ and $\tilde{B}$ are Borel. We will
need the following easy consequence of the Arsenin-Kunugui theorem
(the proof is left to the reader).
\begin{lem}
\label{ak} Let $Z$ be a standard Borel space, $X$ a Polish space and
$F\subseteq Z\times X$ Borel with $K_\sg$ sections. Then the set
$\tilde{F}$ defined by
\[ (z,x)\in \tilde{F} \Leftrightarrow x \in \overline{F_z} \]
is a Borel subset of $Z\times X$.
\end{lem}
By our assumptions, for every $L\in C$ the function $f_L$ is Baire-1
and so for every $(L,K)\in C\times K(X)$ the sections $A_{(L,K)}$
and $B_{(L,K)}$ of $A$ and $B$ respectively are $K_\sg$. Hence,
by Lemma \ref{ak}, we get that $\tilde{A}$ and $\tilde{B}$ are Borel.

By the above we conclude that $\mathbb{D}$ is a Borel map. By
Theorem \ref{tp}, the map $(L,K)\to |K|_{\mathbb{D}_L}$ is a
$\PB^1_1$-rank on $\Omega_{\mathbb{D}}$. By Proposition \ref{kl}
and the fact that $C\subseteq \lbf^1$, we get that for every
$(L,K)\in C\times K(X)$ the transfinite sequence
$\big( \mathbb{D}_L^{(\xi)}(K)\big)_{\xi<\omega_1}$ must be
stabilized at $\varnothing$ and so $\Omega_\mathbb{D}=C\times K(X)$.
As $\Omega_{\mathbb{D}}$ is Borel, by boundedness we have
\[ \sup\{ |K|_{\mathbb{D}_L}: (L,K)\in C\times K(X)\}<\omega_1. \]
It follows that
\[ \sup\{ \alpha(f_L,K,a,b): (L,K)\in C\times K(X)\}<\omega_1. \]
By Proposition \ref{p2}, we get
\[ \sup\{ \alpha(f_L,a,b): L\in C\}<\omega_1. \]
This completes the proof of the theorem.
\end{proof}

\subsection{Consequences}

Let us recall some definitions from \cite{ADK}. Let $X$ be a
Polish space, $(f_n)_n$ a sequence of real-valued functions on $X$
and let $\kk$ be the closure of $\{f_n\}_n$ in $\rr^X$. We will
say that $\kk$ is a (separable) quasi-Rosenthal if every accumulation
point of $\kk$ is a Baire-1 function and moreover we will say that
$\kk$ is Borel separable if the sequence $(f_n)_n$ consists of
Borel functions. Combining Theorem \ref{t4} with Corollary \ref{c}
we get the following result of \cite{ADK}.
\begin{thm}
\label{t5} Let $X$ be a Polish space and $\kk$ a Borel separable
quasi-Rosenthal compact on $X$. Then
\[ \sup\{ \alpha(f):f\in Acc(\kk)\}<\omega_1 \]
where $Acc(\kk)$ denotes the accumulation points of $\kk$. In
particular, if $\kk$ is a separable Rosenthal compact on $X$, then
\[ \sup\{ \alpha(f): f\in\kk\}<\omega_1. \]
\end{thm}
Besides boundedness, the implications of Theorem \ref{t4} and the
relation between the separation rank and the Borelness of $\lbf$
are more transparently seen when $X$ is a compact metrizable space.
In particular we have the following.
\begin{prop}
Let $X$ be a compact metrizable space and $\kk$ a separable
Rosenthal compact on $X$. Let $\bs$ be a dense sequence in
$\kk$ and $a<b$ reals. Then the map $L\to \alpha(f_L,a,b)$
is a $\PB^1_1$-rank on $\lbf$ if and only if the set $\lbf$
is Borel.
\end{prop}
\begin{proof}
First assume that $\lbf$ is not Borel. By Theorem \ref{t5},
we have that
\[ \sup\{ \alpha(f_L,a,b): L\in\lbf\}<\omega_1 \]
and so the map $L\to \alpha(f_L,a,b)$ cannot be a $\PB^1_1$-rank
on $\lbf$, as $\lbf$ is $\PB^1_1$-true. Conversely, assume that
$\lbf$ is Borel. By the proof of Theorem \ref{t4}, we have that
the map $(L,K)\to |K|_{\mathbb{D}_L}$ is a $\PB^1_1$-rank on
$\lbf\times K(X)$. It follows that the relation
\[ L_1\preceq L_2 \Leftrightarrow \alpha(f_{L_1},a,b)\leq
\alpha(f_{L_2},a,b)\Leftrightarrow |X|_{\mathbb{D}_{L_1}}
\leq |X|_{\mathbb{D}_{L_2}} \]
is Borel in $\lbf\times \lbf$. This implies that the map
$L\to\alpha(f_L,a,b)$ is a $\PB^1_1$-rank on $\lbf$, as desired.
\end{proof}
\begin{rem}
Although the map $L\to\alpha(f_L,a,b)$ is not always a
$\PB^1_1$-rank on $\lbf$, it is easy to see that it is a
$\PB^1_1$-rank on the codes $C$ of $\kk$, as the relation
\[ c_1\preceq c_2 \Leftrightarrow \alpha(f_{c_1},a,b)\leq
\alpha(f_{c_2},a,b)\Leftrightarrow |X|_{\mathbb{D}_{c_1}}
\leq |X|_{\mathbb{D}_{c_2}} \]
is Borel in $C\times C$ for every pair $a<b$ of reals. Hence,
when $X$ is compact metrizable space, we could say that the
separation rank is a $\PB^1_1$-rank "in the codes".
\end{rem}
We proceed to discuss another application of Theorem \ref{t4} which
deals with the following approximation question in Ramsey theory.
Recall that a set $N\subseteq [\nn]$ is called Ramsey-null
if for every $s\in [\nn]^{<\nn}$ and every $L\in [\nn]$ with $s<L$,
there exists $L'\in [L]$ such that $[s,L']\cap N=\varnothing$.
As every analytic set is Ramsey \cite{S1}, it is natural to ask
the following. Is it true that for every analytic set $A\subseteq
[\nn]$ there exists $B\supseteq A$ Borel such that $B\setminus A$
is Ramsey-null? As we will show the answer is no and a counterexample
can be found which is in addition Ramsey-null.

To this end we will need some notations from \cite{AGR}. Let
$X$ be a separable Banach space. By $X^{**}_{\mathcal{B}_1}$
we denote the set of all Baire-1 elements of the ball of
the second dual $X^{**}$ of $X$. We will say that $X$ is
$\alpha$-universal if
\[ \sup\{ \alpha(x^{**}): x^{**}\in X^{**}_{\mathcal{B}_1} \}=\omega_1. \]
We should point out that there exist non-universal (in the
classical sense) separable Banach spaces which are $\alpha$-universal
(see \cite{AD}). We have the following.
\begin{prop}
\label{rn} There exists a $\SB^1_1$ Ramsey-null subset $A$
of $[\nn]$ for which there does not exist a Borel set $B\supseteq A$
such that the difference $B\setminus A$ is Ramsey-null.
\end{prop}
\begin{proof}
Let $X$ be a separable $\alpha$-universal Banach space and
fix a norm dense sequence $\mathbf{f}=(x_n)_n$ in the ball of
$X$ (it will be convenient to assume that $x_n\neq x_m$ if
$n\neq m$). Let
\[ \lbf=\{ L\in[\nn]: (x_n)_{n\in L} \text{ is weak* convergent} \}. \]
Clearly $\lbf$ is $\PB^1_1$. Moreover, notice that $\lbf=\lbf^1$ according
to the notation of Theorem \ref{t4}.

Let $x^{**}\in X^{**}_{\mathcal{B}_1}$ arbitrary. By the Odell-Rosenthal
theorem (see \cite{AGR} or \cite{OR}), there exists $L\in\lbf$ such that
$x^{**}=w^*-\lim_{n\in L} x_n$. It follows that
\[ \sup\{ \alpha(x^{**}): x^{**}\in X^{**}_{\mathcal{B}_1} \} =
\sup\{ \alpha(x_L): L\in\lbf \} \]
where $x_L$ denotes the weak* limit of the sequence $(x_n)_{n\in L}$.
Denote by $(e_n)_n$ the standard basis of $\ell_1$ and let
\[ \Lambda=\{ L\in [\nn]: \exists k \text{ such that } (x_n)_{n\in L}
\text{ is } (k+1)-\text{equivalent to } (e_n)_n \} \]
where, as usual, if $L\in [\nn]$ with $L=\{ l_0<l_1<...\}$ its increasing
enumeration, then $(x_n)_{n\in L}$ is $(k+1)$-equivalent to $(e_n)_n$ if
for every $m\in \nn$ and every $a_0,...,a_m\in \rr$ we have
\[ \frac{1}{k+1} \sum_{n=0}^m |a_n| \leq \Big\| \sum_{n=0}^m a_nx_{l_n}\Big\|_X
\leq (k+1) \sum_{n=0}^m |a_n|. \]
Then $\Lambda$ is $\SB^0_2$. We notice that, by Bourgain's result
\cite{Bou} and our assumptions on the space $X$, the set $\Lambda$
is non-empty. Let also
\[ \Lambda_1=\{ N\in[\nn]: \exists L\in \Lambda \ \exists s\in[\nn]^{<\nn}
\text{ such that } N\bigtriangleup L=s\}. \]
Clearly $\Lambda_1$ is $\SB^0_2$ too. Observe that both $\lbf$ and $\Lambda_1$
are hereditary and invariant under finite changes. Moreover the set $\lbf
\cup \Lambda_1$ is cofinal. This is essentially a consequence of Rosenthal's
dichotomy (see, for instance, \cite{LT}). It follows that the set
$A=[\nn]\setminus (\lbf\cup \Lambda_1)$ is $\SB^1_1$ and Ramsey-null.

We claim that $A$ is the desired set. Assume not, i.e. there exists a
Borel set $B\supseteq A$ such that the difference $B\setminus A$ is Ramsey-null.
We set $C=[\nn]\setminus (B\cup\Lambda_1)$. Then $C\subseteq\lbf$ is Borel
and moreover $\lbf\setminus C$ is Ramsey-null. It follows that for every
$x^{**}\in X^{**}_{\mathcal{B}_1}$ there exists $L\in C$ such that
$x^{**}=x_L$. As $C$ is Borel, by Theorem \ref{t4} we have that
\[ \sup\{\alpha(x^{**}): x^{**}\in X^{**}_{\mathcal{B}_1} \} =
\sup\{ \alpha(x_L): L\in C \}<\omega_1 \]
which contradicts the fact that $X$ is $\alpha$-universal. The proof
is completed.
\end{proof}
\begin{rem}
(1) An example as in Proposition \ref{rn} can also be given using
the convergence rank $\gamma$ studied by A. S. Kechris and A. Louveau
\cite{KL}. As the reasoning is the same, we shall briefly describe
the argument. Let $(f_n)_n$ be a sequence of continuous function
on $2^\nn$ with $\|f_n\|_\infty\leq 1$ for all $n\in\nn$ and such
that the set $\{f_n:n\in\nn\}$ is norm dense in the ball
of $C(2^\nn)$. As in Proposition \ref{rn}, consider the sets
$\lbf$, $\Lambda_1$ and $A=[\nn]\setminus(\lbf\cup \Lambda_1)$.
Then the set $A$ is $\SB^1_1$ and Ramsey-null. That $A$ cannot
be covered by a Borel set $B$ such that the difference
$B\setminus A$ is Ramsey-null follows essentially by the following facts.
\begin{enumerate}
\item[(F1)] The map $(g_n)_n\mapsto \gamma\big((g_n)_n\big)$ is
a $\PB^1_1$-rank on the set $\mathrm{CN}=\{ (g_n)_n\in C(2^\nn)^\nn:
(g_n)_n \text{ is pointwise convergent}\}$ (see \cite{Kechris},
page 279). Hence the map
\[ \lbf\ni L=\{l_0<l_1<...\}\mapsto \gamma\big( (f_{l_n})_n\big) \]
is a $\PB^1_1$-rank on $\lbf$.
\item[(F2)] For every $\Delta\in \mathbf{\Delta}^0_2$, there
exists $L\in\lbf$ such that the sequence $(f_n)_{n\in L}$ is
pointwise convergent to $\chi_\Delta$. By Proposition
1 in \cite{KL}, we get that $\alpha(\chi_\Delta)\leq \gamma
\big( (f_n)_{n\in L}\big)$. It follows that
\[ \sup\{ \gamma\big((f_n)_{n\in L}\big): L\in\lbf\}\geq \sup\{
\alpha(\chi_{\Delta}): \Delta\in \mathbf{\Delta}^0_2\}=\omega_1. \]
\end{enumerate}
\noindent (2) For the important special case of a separable Rosenthal
compact $\kk$ defined on a compact metrizable space $X$ and having
a dense set of continuous functions, Theorem \ref{t5} has originally
been proved by J. Bourgain \cite{Bou}. We should point out that in
this case one does not need Corollary \ref{c} in order to carry
out the proof. Let us briefly explain how this can be done.
So assume that $X$ is compact metrizable and $\bs$ is a sequence
of continuous functions dense in $\kk$. Fix $a,b\in\qq$ with $a<b$
and let $A_n=[f_n\leq a]$ and $B_n=[f_n\geq b]$. For a given
$M\in [\nn]$ let as usual
\[ \liminf_{n\in M} A_n = \bigcup_n \bigcap_{k\geq n, k\in M} A_k \]
and similarly for $\liminf_{n\in M} B_n$. Observe the following.
\begin{enumerate}
\item[(O1)] For every $M\in[\nn]$ the sets $\liminf_{n\in M} A_n$
and $\liminf_{n\in M} B_n$ are both $\SB^0_2$.
\item[(O2)] If $L,M\in[\nn]$ are such that $L\subseteq M$, then
$\liminf_{n\in M} A_n\subseteq \liminf_{n\in L} A_n$ and similarly
for $B_n$.
\item[(O3)] If $L\in \lbf$, then $[f_L<a]\subseteq \liminf_{n\in L}
A_n \subseteq [f_L\leq a]$ and respectively $[f_L>b]\subseteq
\liminf_{n\in L} B_n \subseteq [f_L\geq b]$.
\end{enumerate}
Define $\mathbb{D}:[\nn]\times K(X)\to K(X)$ by
\[ \mathbb{D}(M,K)= \overline{ K\cap\liminf_{n\in M} A_n}
\cap \overline{ K\cap\liminf_{n\in M} B_n}.\]
By (O1) and using the same arguments as in the proof of Theorem
\ref{t4}, we can easily verify that $\mathbb{D}$ is Borel. As
$\lbf$ is cofinal, by (O2) and (O3) we can also easily verify that
$\Omega_{\mathbb{D}}=[\nn]\times K(X)$. So by boundedness we get
$\sup\{ |K|_{\mathbb{D}_M}: (M,K)\in[\nn]\times K(X)\}<\omega_1$.
Now using (O3) again, we finally get that $\sup\{ \alpha(f,a,b):
f\in \kk\}<\omega_1$, as desired.
\end{rem}


\section{On the descriptive set-theoretic properties of $\lbf$}

In this section we will show that certain topological properties of
a separable Rosenthal compact $\kk$ imply the Borelness of the set
$\lbf$. To this end, we recall that $\kk$ is said to be a pre-metric
compactum of degree at most two if there exists a countable subset
$D$ of $X$ such that at most two functions in $\kk$ agree on $D$
(see \cite{To}). Let us consider the following subclass.
\begin{defn}
We will say that $\kk$ is a pre-metric compactum of degree exactly two,
if there exist a countable subset $D$ of $X$ and a countable subset
$\mathcal{D}$ of $\kk$ such that at most two functions in $\kk$ coincide
on $D$ and moreover for every $f\in\kk\setminus\mathcal{D}$ there exists
$g\in\kk$ with $g\neq f$ and such that $g$ coincides with $f$ on $D$.
\end{defn}
An important example of such a compact is the split interval (but it is
not the only important one -- see Remark \ref{r2} below). Under the
above terminology we have the following.
\begin{thm}
\label{t3} Let $X$ be a Polish space and $\kk$ a separable Rosenthal
compact on $X$. If $\kk$ is pre-metric of degree exactly two, then
for every dense sequence $\bs$ in $\kk$ the set $\lbf$ is Borel.
\end{thm}
\begin{proof}
Let $\bs$ be a dense sequence in $\kk$ and $C$ be the set of codes
of $(f_n)_n$. Let also $D\subseteq X$ countable and $\mathcal{D}
\subseteq\kk$ countable verifying that $\kk$ is pre-metric of degree
exactly two.
\bigskip

\noindent \textsc{Claim.} \textit{There exists $\mathcal{D}'\subseteq \kk$
countable with $\mathcal{D}\subseteq\mathcal{D}'$ and such that for every
$c\in C$ with $f_c\in\kk\setminus \mathcal{D}'$ there exists $c'\in C$ such
that $f_{c'}\neq f_c$ and $f_{c'}$ coincides with $f_c$ on $D$.}
\bigskip

\noindent \textit{Proof of the claim.} Let $c\in C$ be such that
$f_c\in\kk\setminus\mathcal{D}$. Let $g$ be the (unique) function
in $\kk$ with $g\neq f_c$ and such that $g$ coincides with $f_c$ on $D$.
If there does not exist $c'\in C$ with $g=f_{c'}$, then $g$ is an isolated
point of $\kk$. We set
\[ \mathcal{D}'=\mathcal{D} \cup \{ f\in\kk: \exists g\in\kk \text{ isolated
such that } f(x)=g(x) \ \forall x\in D\}. \]
As the isolated points of $\kk$ are countable and $\kk$ is pre-metric
of degree at most two, we get that $\mathcal{D}'$ is countable. Clearly
$\mathcal{D}'$ is as desired.  \hfill $\lozenge$
\bigskip

\noindent Let $\mathcal{D}'$ be the set obtained above and put
\[ \ld=\bigcup_{f\in\mathcal{D}'} \llf= \bigcup_{f\in\mathcal{D}'}
\{ L\in [\nn]: (f_n)_{n\in L} \text{ is pointwise convergent to } f\}. \]
As every point in $\kk$ is $G_\delta$, we see that $\ld$ is Borel (actually
it is $\SB^0_4$). Consider the following equivalence relation $\sim$ on $C$,
defined by
\[ c_1\sim c_2 \Leftrightarrow \forall x\in D \ f_{c_1}(x)=f_{c_2}(x). \]
By Lemma \ref{cl1}, the equivalence relation $\sim$ is Borel. Consider
now the relation $P$ on
$C\times C\times K(X)\times X$ defined by
\[ (c_1,c_2,K,x)\in P \Leftrightarrow (c_1\sim c_2) \ \wedge \ (x\in K) \
\wedge \ (|f_{c_1}(x)-f_{c_2}(x)|>0). \]
Again $P$ is Borel. Moreover notice that for every $c_1,c_2\in C$ the function
$x\mapsto |f_{c_1}(x)-f_{c_2}(x)|$ is Baire-1, and so, for every $(c_1,c_2,K)
\in C\times C\times K(X)$ the section $P_{(c_1,c_2,K)}=\{ x\in X: (c_1,c_2,K,x)
\in P\}$ of $P$ is $K_\sg$. By Theorem 35.46 in \cite{Kechris}, the set
$S\subseteq C\times C\times K(X)$ defined by
\[ (c_1,c_2,K)\in S \Leftrightarrow \exists x \ (c_1,c_2,K,x)\in P \]
is Borel and there exists a Borel map $\phi:S\to X$ such that for every
$(c_1,c_2,K)\in S$ we have $\big(c_1,c_2,K,\phi(c_1,c_2,K)\big)\in P$.
By the above claim, we have that for every $c\in C\setminus\ld$ there
exist $c'\in C$ and $K\in K(X)$ such that $(c,c',K)\in S$. Moreover,
observe that the set $D\cup \big\{\phi(c,c',K)\big\}$ determines the
neighborhood basis of $f_c$. The crucial fact is that this can be done
in a Borel way.

Now we claim that
\begin{eqnarray*}
L\in \lbf & \Leftrightarrow & (L\in\ld) \vee \Big[ (\forall x\in D \
\big( f_n(x) \big)_{n\in L} \text{ converges}) \ \wedge \\
& & \big\{ \exists s\in S \text{ with } s=(c_1,c_2,K) \text{ such that } \\
& & \ [\forall x\in D \ f_{c_1}(x)=\lim_{n\in L} f_n(x)] \ \wedge \\
& & \ [\big( f_n(\phi(s))\big)_{n\in L} \text{ converges}] \ \wedge \\
& & \ [f_{c_1}(\phi(s))= \lim_{n\in L} f_n(\phi(s))] \big\} \Big].
\end{eqnarray*}
Grating this, the proof is completed as the above expression gives a
$\SB^1_1$ definition of $\lbf$. As $\lbf$ is also $\PB^1_1$, this
implies that $\lbf$ is Borel, as desired.

It remains to prove the above equivalence. First assume that $L\in\lbf$.
We need to show that $L$ satisfies the expression on the right. If
$L\in\ld$ this is clearly true. If $L\notin\ld$, then pick a code
$c\in C\setminus\ld$ such that $f_c=f_L$. By the above claim and
the remarks of the previous paragraph we can easily verify that
again $L$ satisfies the expression on the right. Conversely, let
$L$ fulfil the right side of the equivalence. If $L\in \ld$ we are
done. If not, then by the Bourgain-Fremlin-Talagrand theorem,
it suffices to show that all convergent subsequences of
$(f_n)_{n\in L}$ have the same limit. The first two conjuncts
enclosed in the square brackets on the right side of the
equivalence guarantee that each such convergent subsequence
of $(f_n)_{n\in L}$ converges either to $f_{c_1}$ or to $f_{c_2}$.
The last two conjuncts guarantee that it is not $f_{c_2}$,
so it is always $f_{c_1}$. Thus $L\in \lbf$ and the proof is completed.
\end{proof}
\begin{rem}
\label{r2} (1) Let $\kk$ be a pre-metric compactum of degree at
most two and let $D\subseteq X$ countable such that at most two
functions in $\kk$ coincide on $D$. Notice that the set $C$ of
codes of $\kk$ is naturally divided into two parts, namely
\[ C_2=\{ c\in C: \exists c'\in C \text{ with } f_c\neq f_{c'}
\text{ and } f_c(x)=f_{c'}(x) \ \forall x\in D\} \]
and its complement $C_1=C\setminus C_2$. The assumption that
$\kk$ is pre-metric of degree exactly two, simply reduces to
the assumption that the functions coded by $C_1$ are at most
countable. We could say that $C_1$ is the set of metrizable codes,
as it is immediate that the set $\{ f_c:c\in C_1\}$ is a metrizable
subspace of $\kk$. It is easy to check, using the set $S$ defined
in the proof of Theorem \ref{t3}, that $C_2$ is always $\SB^1_1$.
As we shall see, it might happen that $C_1$ is $\PB^1_1$-true.
However, if the set $C_1$ of metrizable codes is Borel, or equivalently
if $C_2$ is Borel, then the set $\lbf$ is Borel too. Indeed, let $\Phi$
be the second part of the disjunction of the expression in the proof
of Theorem \ref{t3}. Then it is easy to see, using the same arguments
as in the proof of Theorem \ref{t3}, that
\[ L\in\lbf \Leftrightarrow (L\in\Phi) \ \vee \ (\exists c\in C_1 \
\forall x\in D \ f_c(x)=\lim_{n\in L} f_n(x) ).\]
Clearly the above formula gives a $\SB^1_1$ definition of $\lbf$,
provided that $C_1$ is Borel.\\
(2) Besides the split interval, there exists another important example
of a separable Rosenthal compact which is pre-metric of degree exactly
two. This is the separable companion of the Alexandroff duplicate of
the Cantor set $D(2^\nn)$ (see \cite{To} for more details). An interesting
feature of this compact is that it is \textit{not} hereditarily separable.
\end{rem}
\begin{examp}
\label{e1} We proceed to give examples of pre-metric compacta of degree
at most two for which Theorem \ref{t3} is not valid. Let us recall first
the split Cantor set $S(2^\nn)$, which is simply the combinatorial
analogue of the split interval. In the sequel by $\leqslant$ we
shall denote the lexicographical ordering on $2^\nn$ and by $<$
its strict part. For every $x\in 2^\nn$ let $f^+_x=
\chi_{\{y:x\leqslant y\}}$ and $f^-_x=\chi_{\{y:x<y\}}$. The split
Cantor set $S(2^\nn)$ is $\{ f^+_x:x\in 2^\nn\}\cup \{ f^-_x: x\in 2^\nn\}$.
Clearly $S(2^\nn)$ is a hereditarily separable Rosenthal compact and
it is a fundamental example of a pre-metric compactum of degree at most
two (see \cite{To}). There is a canonical dense sequence in $S(2^\nn)$
defined as follows. Fix a bijection $h:\ct\to\nn$ such that $h(s)<h(t)$
if  $|s|<|t|$ and enumerate the nodes of Cantor tree as $(s_n)_n$ according
to $h$. For every $s\in\ct$ let $x^0_s=s^\con 0^{\infty}\in 2^\nn$ and
$x^1_s=s^\con 1^{\infty}\in 2^\nn$. For every $n\in\nn$ let $f_{4n}=
f^+_{x^0_{s_n}}$, $f_{4n+1}=f^+_{x^1_{s_n}}$, $f_{4n+2}=f^-_{x^0_{s_n}}$
and $f_{4n+3}=f^-_{x^1_{s_n}}$. The sequence $(f_n)_{n\in \nn}$ is a
dense sequence in $S(2^\nn)$.

Let $A$ be a subset of $2^\nn$ such that $A$ does not contain the eventually
constant sequences. To every such $A$ one associates naturally a subset
of $\rr^A$, which we will denote by $S(A)$, by simply restricting every
function of $S(2^\nn)$ on $A$. Clearly if $A$ is $\SB^1_1$, then $S(A)$
is again a hereditarily separable Rosenthal compact. Notice however
that if $2^\nn\setminus A$ is uncountable, then $S(A)$ is not of degree
exactly two. The dense sequence $(f_n)_n$ of $S(2^\nn)$ still remains
a dense sequence in $S(A)$. Viewing $(f_n)_n$ as a dense sequence in $S(A)$,
we let
\[ \mathcal{L}_A=\{ L\in[\nn]: (f_n|_A)_{n\in L} \text{ is pointwise
convergent on } A \}. \]
Under the above notations we have the following.
\begin{prop}
\label{p1} Let $A\subseteq 2^\nn$ be $\SB^1_1$ and such that $A$ does not
contain the eventually constant sequences. Then $2^\nn\setminus A$ is Wadge
reducible to $\mathcal{L}_A$. In particular, if $A$ is $\SB^1_1$-complete,
then $\mathcal{L}_A$ is $\PB^1_1$-complete. Moreover, if $A$ is Borel,
then $\mathcal{L}_A$ is Borel too.
\end{prop}
\begin{proof}
Consider the map $\Phi:2^\nn\to 2^{\ct}$, defined by
\[ \Phi(x)=\big\{ x(0)+1, x(0)^\con (x(1)+1), x(0)^\con x(1)^\con (x(2)+1),
... \big\} \]
where the above addition is taken modulo 2. Clearly $\Phi$ is continuous.
Let $h:\ct\to \nn$ be the fixed enumeration of the nodes of the Cantor tree.
For every $x\in 2^\nn$ we put $L_x=\{ h(t): t\in \Phi(x)\}\in[\nn]$ and
also for every $t\in \Phi(x)$ let  $x^0_t=t^\con 0^{\infty}\in 2^\nn$.
Notice that $(t_n)_{n\in L_x}$ is the enumeration of $\Phi(x)$ according
to $h$ and moreover that $x$ is the limit of the sequence $(x^0_{t_n})_{n\in L_x}$.
However, as one easily observes, if $x$ is not eventually constant, then
there exist infinitely many $n\in L_x$ such that $x^0_{t_n}< x$ and
infinitely many $n\in L_x$ such that $x< x^0_{t_n}$.

Now define $H: 2^\nn\to [\nn]$ by
\[ H(x)=\{ 4h(t): t\in \Phi(x) \}=\{ 4n :n\in L_x\}. \]
Clearly $H$ is continuous. We claim that
\[ x\notin A \Leftrightarrow H(x)\in \mathcal{L}_A . \]
Indeed, first assume that $x\notin A$. As we have remarked before,
we have that $x=\lim_{n\in L_x} x^0_{t_n}$.
Notice that the sequence $(f_n)_{n\in H(x)}$ is simply the sequence
$\big( f^+_{x^0_t} \big)_{t\in \Phi(x)}$. Observe that for
every $y\neq x$ the sequence $\big( f_n(y)\big)_{n\in H(x)}$
converges to $0$ if $y<x$ and to $1$ if $x<y$. As $x\notin A$, this
implies that $(f_n)_{n\in H(x)}$ is pointwise convergent, and so
$H(x)\in\mathcal{L}_A$. Conversely, assume that $x\in A$. As $A$
does not contain the eventually constant sequences, by the remarks
after the definition of $\Phi$, we get that there exist infinitely many
$n\in H(x)$ such that $f_n(x)=0$ and infinitely many $n\in H(x)$
such that $f_n(x)=1$. Hence the sequence $\big(f_n(x)\big)_{n\in H(x)}$
does not converge, and as $x\in A$, we conclude that $H(x)\notin
\mathcal{L}_A$. As $H$ is continuous, this completes the proof
the proof that $2^\nn\setminus A$ is Wadge reducible to $\mathcal{L}_A$.
Finally, the fact that if $A$ is Borel, then $\mathcal{L}_A$ is Borel too
follows by straightforward descriptive set-theoretic computation
and we prefer not to bother the reader with it.
\end{proof}
\end{examp}
\begin{rem}
Besides the fact that Theorem \ref{t3} cannot be lifted to all pre-metric
compacta of degree at most two, Proposition \ref{p1} has another consequence.
Namely that we cannot bound the Borel complexity of $\lbf$ for a dense
sequence $\bs$ in $\kk$. This is in contrast with the situation with $\llf$
for some $f\in\kk$, which when it is Borel (equivalently when $f$ is a
$G_\delta$ point), it is always $\PB^0_3$.
\end{rem}
Concerning the class of not first countable separable Rosenthal
compacta we have the following.
\begin{prop}
\label{newp1} Let $\kk$ be a separable Rosenthal compact.
If there exists a non-$G_\delta$ point $f$ in $\kk$, then
for every dense sequence $\bs$ in $\kk$ the set $\lbf$
is $\PB^1_1$-complete.
\end{prop}
The proof of Proposition \ref{newp1} is essentially based on
a result of A. Krawczyk from \cite{Kra}. To state it, we need to
recall some pieces of notation and few definitions. For every
$a,b\in [\nn]$ we write  $a\subseteq^* b$ if $a\setminus b$
is finite, while we write $a\perp b$ if the set $a\cap b$
is finite. If $\mathcal{A}$ is a subset of $[\nn]$, we let
$\mathcal{A}^\perp=\{ b: b\perp a \ \forall a\in \mathcal{A}\}$
and $\mathcal{A}^*= \{ \nn\setminus a: a\in\mathcal{A}\}$.
For every $\mathcal{A}, \mathcal{B} \subseteq [\nn]$ we say
that $\mathcal{A}$ is countably $\mathcal{B}$-generated if there
exists $\{b_n:n\in\nn\}\subseteq \mathcal{B}$ such that for
every $a\in\mathcal{A}$ there exists $k\in\nn$ with
$a\subseteq b_0\cup ... \cup b_k$. An ideal $\ii$ on $\nn$ is
said to be bi-sequential if for every $p\in \beta\nn$ with
$\ii\subseteq p^*$, $\ii$ is countably $p^*$-generated.
Finally, for every $t\in\bt$ let $\hat{t}=\{ s: t\sqsubset s\}$.
We will use the following result of Krawczyk (see \cite{Kra},
Lemma 2).
\begin{prop}
\label{newp2} Let $\ii$ be $\SB^1_1$, bi-sequential and not
countably $\ii$-generated ideal on $\nn$. Then there exists
a 1-1 map $\psi:\bt\to\nn$ such that, setting $\jj=\{ \psi^{-1}(a):
a\in\ii\}$, the following hold.
\begin{enumerate}
\item[(P1)] For every $\sg\in\nn^\nn$, $\{ \sg|n: n\in\nn\}\in \jj$.
\item[(P2)] For every $b\in\jj$ and every $n\in\nn$, there exist
$t_0,...,t_k\in \nn^n$ with $b\subseteq ^* \hat{t}_0 \cup ...
\cup \hat{t}_k$.
\end{enumerate}
\end{prop}
We continue with the proof of Proposition \ref{newp1}.
\begin{proof}[Proof of Proposition \ref{newp1}]
Let $\bs$ be a dense sequence in $\kk$ and let $f\in\kk$
be a non-$G_\delta$ point. Consider the ideal
\[ \ii=\{ L\in [\nn]: f\notin \overline{\{f_n\}}^p_{n\in L}\}. \]
In \cite{Kra}, page 1099, it is shown that $\ii$ is a $\SB^1_1$,
bi-sequential ideal on $\nn$ which is not countably
$\ii$-generated (the bi-sequentiality of $\ii$ can be
derived either by a result of Pol \cite{Pol3}, or by the
non-effective version of Debs' theorem \cite{AGR}).
Also, by the Bourgain-Fremlin-Talagrand theorem, we have
that $\ii^\perp=\llf$. We apply Proposition \ref{newp2} and we
get a 1-1 map $\psi:\bt\to\nn$ satisfying (P1) and (P2).
\bigskip

\noindent \textsc{Claim.} \textit{For every $T\in\wf$ infinite,
$T\in \jj^\perp$.}
\bigskip

\noindent \textit{Proof of the claim.} Assume not. Then there exist
$T\in\wf$ infinite and $b\in\jj$ with $b\subseteq T$. For every
$s\in T$ let $T_s=\{ t\in T: s\sqsubseteq t\}$. We let $S=\{ s\in T:
T_s\cap b \text{ is infinite}\}$. Then $S$ is downwards closed
subtree of $T$. Moreover, by (P2) in Proposition \ref{newp2}, we see that
$S$ is finitely splitting. Finally, notice that for every $s\in S$
there exists $n\in\nn$ with $s^{\con}n\in S$. Indeed, let
$s\in S$ and put $b_s=T_s\cap b\in\jj$. Let $N_s=\{ n\in\nn:
s^{\con} n\in T\}$ and observe that $b_s\setminus \{s\}=
\bigcup_{n\in N_s} (T_{s^{\con}n}\cap b_s)$. By (P2) in
Proposition \ref{newp2} again, we get that there exists $n_0\in N_s$
with $T_{s^{\con}n_0}\cap b_s$ infinite. Thus $s^{\con}n_0\in S$. It
follows that $S$ is a finitely splitting, infinite tree. By K\"{o}nig's
Lemma, we see that $S\notin\wf$. But then $T\notin \wf$, a contradiction.
The claim is proved. \hfill $\lozenge$
\bigskip

\noindent Fix $T_0\in\wf$ infinite. The map $\Psi:\tr\to [\nn]$ defined
by  $\Psi(T)=\{ \psi(t): t\in T\cup T_0\}$ is clearly continuous.
If $T\in\wf$, then $T\cup T_0\in\wf$. By the above claim, we see that
$T\cup T_0\in \jj^\perp$, and so, $\Psi(T)\in \ii^\perp=\llf\subseteq \lbf$.
On the other hand, if $T\notin \wf$, then by (P1) in Proposition
\ref{newp2} and the above claim, we get that there exist $L\in\lbf\setminus
\llf$ and $M\in\llf$ with $L\cup M\subseteq \Psi(T)$. Hence
$\Psi(T)\notin \lbf$. It follows that $\wf$ is Wadge reducible to $\lbf$
and the proof is completed.
\end{proof}
The last part of this section is devoted to the construction of
canonical $\PB^1_1$-ranks on the sets $\lbf$ and $\llf$. So let
$X$ be a Polish space, $\bs$ a sequence relatively compact in
$\mathcal{B}_1(X)$ and $f$ an accumulation point of $(f_n)_n$.
As the sets $\lbf$ and $\llf$ do not depend on the topology on
$X$, we may assume, by enlarging the topology of $X$ if necessary,
that the functions $(f_n)_n$ and the function $f$ are continuous
(see \cite{Kechris}). We need to deal with decreasing sequences
of closed subsets of $X$ \`{a} la Cantor. We fix a countable dense
subset $D$ of $X$. Let $(B_n)_n$ be an enumeration of all closed
balls in $X$ (taken with respect to some compatible complete metric)
with centers in $D$ and rational radii. If $X$ happens to be
locally compact, we will assume that every ball $B_n$ is compact.
We will say that a finite sequence $w=(l_0,...,l_k)\in\bt$
is acceptable if
\begin{enumerate}
\item[(i)] $B_{l_0}\supseteq B_{l_1}\supseteq ... \supseteq B_{l_k}$, and
\item[(ii)] $\text{diam}(B_{l_i})\leq\frac{1}{i+1}$ for all $i=0,...,k$.
\end{enumerate}
By convention $(\varnothing)$ is acceptable. Notice that if $w_1\sqsubset w_2$
and $w_2$ is acceptable, then $w_1$ is acceptable too. We will also need the
following notations.
\begin{notation}
By $\fin$ we denote the set of all finite subsets of $\nn$. For every
$F,G\in \fin$ we write $F<G$ if $\max\{n: n\in F\}<\min\{n:n\in G\}$.
For every $L\in [\nn]$, by $\fin(L)$ we denote the set of all finite
subsets of $L$. Finally, by $[\fin(L)]^{<\nn}$ we denote the set
of all finite sequences $t=(F_0,...,F_k)\in \big(\fin(L)\big)^{<\nn}$
which are increasing, i.e. $F_0<F_1<...<F_k$.
\end{notation}
The construction of the $\PB^1_1$-ranks on $\lbf$ and $\llf$ will be
done by finding appropriate reductions of the sets in question to
well-founded trees. In particular, we shall construct the following.
\begin{enumerate}
\item[(C1)] A continuous map $[\nn]\ni L\mapsto T_L\in \tr(\nn\times
\fin\times\nn)$, and
\item[(C2)] a continuous map $[\nn]\ni L\mapsto S_L\in \tr(\nn\times\nn)$
\end{enumerate}
such that
\begin{enumerate}
\item[(C3)] $L\in\lbf$ if and only if $T_L\in\wf(\nn\times\fin\times\nn)$, and
\item[(C4)] $L\in\llf$ if and only if $S_L\in\wf(\nn\times\nn)$.
\end{enumerate}
It follows by (C1)-(C4) above, that the maps $L\to o(T_L)$ and $L\to o(S_L)$
are $\PB^1_1$-ranks on $\lbf$ and $\llf$.
\bigskip

\noindent 1. \textit{The reduction of $\lbf$ to $\wf(\nn\times\fin\times\nn)$}.
Let $d\in\nn$. For every $L\in[\nn]$ we associate a tree $T^d_L\in\tr(\nn
\times\fin\times\nn)$ as follows. We let
\begin{eqnarray*}
T^d_L= \big\{ (s,t,w) & : & \exists k \text{ with } |s|=|t|=|w|=k, \\
& & s=(n_0<...<n_{k-1})\in [L]^{<\nn}, \\
& & t=(F_0<...<F_{k-1})\in [\fin(L)]^{<\nn},  \\
& & w=(l_0,...,l_{k-1})\in\bt \text{ is acceptable and} \\
& & \forall  0\leq i\leq k-1 \ \forall z\in B_{l_i} \text{ there
exists } m_i\in F_i \text{ with } \\
& & |f_{n_i}(z)-f_{m_i}(z)|>\frac{1}{d+1}\big\}.
\end{eqnarray*}
Next we glue the sequence of trees $(T^d_L)_{d\in\nn}$ in a natural way
and we build a tree $T_L\in\tr(\nn\times\fin\times\nn)$ defined by the rule
\begin{eqnarray*}
(s,t,w)\in T_L & \Leftrightarrow & \exists d \ \exists (s',t',w') \text{ such that }
(s',t',w')\in T^d_L \text{ and}\\
& & s=d^\con s', t=\{d\}^\con t', w=d^\con w'.
\end{eqnarray*}
It is clear that the map $[\nn]\ni L\mapsto T_L\in\tr(\nn\times\fin\times\nn)$
is continuous. Moreover the following holds.
\begin{lem}
\label{l1} Let $L\in[\nn]$. Then $L\in\lbf$ if and only if
$T_L\in\wf(\nn\times\fin\times\nn)$.
\end{lem}
\begin{proof}
First, notice that if $L\notin\lbf$, then there exist $L_1, L_2\in[L]$
such that $L_1\cap L_2=\varnothing$, $L_1, L_2\in\lbf$ and
$f_{L_1}\neq f_{L_2}$, where as usual $f_{L_1}$ and $f_{L_2}$
are the pointwise limits of the sequences $(f_n)_{n\in L_1}$ and
$(f_n)_{n\in L_2}$ respectively. Pick $x\in X$ and $d\in\nn$ such that
$|f_{L_1}(x)-f_{L_2}(x)|>\frac{1}{d+1}$. Clearly we may assume that
$|f_n(x)-f_m(x)|>\frac{1}{d+1}$ for every $n\in L_1$ and every $m\in L_2$.
Let $L_1=\{n_0<n_1<...\}$ and $L_2=\{m_0<m_1<...\}$ be the increasing
enumerations of $L_1$ and $L_2$. Using the continuity of the functions
$(f_n)_n$, we find  $w=(l_0,...,l_k,...)\in\nn^\nn$ such that $w|k$ is
acceptable for all $k\in\nn$, $\bigcap_k B_{l_k}=\{x\}$ and $|f_{n_k}(z)-
f_{m_k}(z)|>\frac{1}{d+1}$ for all $k\in\nn$ and $z\in B_{l_k}$.
Then $\big((n_0,...,n_k),(\{m_0\},...,\{m_k\}), w|k\big)\in T^d_L$ for all
$k\in\nn$, which shows that $T_L\notin \wf(\nn\times\fin\times\nn)$.

Conversely assume that $T_L$ is not well-founded. There exists $d\in\nn$
such that $T^d_L$ is not well-founded too. Let $\big((s_k,t_k,w_k)\big)_k$
be an infinite branch of $T^d_L$. Let $N=\bigcup_k s_k=\{n_0<...<n_k<...\}
\in [L]$, $\mathcal{F}=\bigcup_k t_k=(F_0<...<F_k<...)\in \fin(L)^\nn$ and
$w=\bigcup_k w_k=(l_0,...,l_k,...)\in \nn^\nn$. By the definition of $T^d_L$,
we get that $\bigcap_k B_{l_k}=\{x\}\in X$ and that for every $k\in\nn$
there exists $m_k\in F_k\subseteq L$ with $|f_{n_k}(x)-f_{m_k}(x)|>\frac{1}{d+1}$.
As $F_i<F_j$ for all $i<j$, we see that $m_i<m_j$ if $i<j$. It follows that
$M=\{m_0<...<m_k<...\}\in [L]$. Thus, the sequence $\big( f_n(x)\big)_{n\in L}$
is not Cauchy and so $L\notin \lbf$, as desired.
\end{proof}
By Lemma \ref{l1}, the reduction of $\lbf$ to $\wf(\nn\times\fin\times\nn)$
is constructed. Notice that for every $L\in\lbf$ and every $d_1\leq d_2$
we have $o(T^{d_1}_L)\leq o(T^{d_2}_L)$ and moreover
$o(T_L)=\sup\{ o(T^d_L):d\in\nn\}+1$.
\begin{rem}
We should point out that the reason why in the definition of
$T^d_L$ the node $t$ is a finite sequence of finite sets rather
than natural numbers, is to get the estimate in Proposition
\ref{newp3} below. Having natural numbers instead of finite
sets would also lead to a canonical rank.
\end{rem}

\noindent 2. \textit{The reduction of $\llf$ to $\wf(\nn\times\nn)$}.
The reduction is similar to that of the previous step, and so, we shall
indicate only the necessary changes. Let $d\in\nn$. As before, for every
$L\in[\nn]$ we associate a tree $S^d_L\in\tr(\nn\times\nn)$ as follows.
We let
\begin{eqnarray*}
S^d_L= \big\{ (s,w) & : & \exists k\in\nn \text{ with } |s|=|w|=k, \\
& & s=(n_0<...<n_{k-1})\in [L]^{<\nn},  \\
& & w=(l_0,...,l_{k-1})\in\bt \text{ is acceptable and} \\
& & \forall 0\leq i\leq k-1 \ \forall z\in B_{l_i} \text{ we have }
|f_{n_i}(z)-f(z)|>\frac{1}{d+1} \big\}.
\end{eqnarray*}
Next we glue the sequence of trees $(S^d_L)_{d\in\nn}$ as we did with the
sequence $(T^d_L)_{d\in\nn}$ and we build a tree $S_L\in\tr(\nn\times\nn)$
defined by the rule
\begin{eqnarray*}
(s,w)\in S_L & \Leftrightarrow & \exists d \ \exists (s',w') \text{ such that }
(s',w')\in T^d_L \text{ and}\\
& & s=d^\con s', w=d^\con w'.
\end{eqnarray*}
Again it is easy to check that the map $[\nn]\ni L\mapsto S_L\in\tr(\nn
\times\nn)$ is continuous. Moreover we have the following analogue of Lemma
\ref{l1}. The proof is identical and is left to the reader.
\begin{lem}
\label{l2} Let $L\in[\nn]$. Then $L\in\llf$ if and only if $S_L\in\wf(\nn\times\nn)$.
\end{lem}
This gives us the reduction of $\llf$ to $\wf(\nn\times\nn)$. As before we have
$o(S_L)=\sup\{ o(S^d_L):d\in\nn\}+1$ for every $L\in\llf$.

We proceed now to discuss the question whether for a given $L\in\llf$
we can bound the order of the tree $S_L$ by the order of $T_L$. The
following example shows that this is not in general possible.
\begin{examp}
\label{e2} Let $A(2^\nn)=\{\delta_\sg:\sg\in 2^\nn\}\cup \{0\}$ be the
one point compactification of $2^\nn$. This is not a separable Rosenthal
compact, but it can be supplemented to one in a standard way (see
\cite{Pol}, \cite{Ma}, \cite{To}). Specifically, let $(s_n)_n$
be the enumeration of the Cantor tree $\ct$ as in Example \ref{e1}.
For every $n\in\nn$, let $f_n=\chi_{V_{s_n}}$, where $V_{s_n}=
\{\sg\in 2^\nn: s_n\sqsubset \sg\}$. Then $A(2^\nn)\cup\{ f_n\}_n$
is a separable Rosenthal compact. Now let $A$ be a $\SB^1_1$ non-Borel
subset of $2^\nn$. Following \cite{Pol2} (see also \cite{Ma}), let
$\kk_A$ be the separable Rosenthal compact obtained by restricting
every function in $A(2^\nn)\cup\{f_n\}_n$ on $A$. The sequence
$\mathbf{f}_A=(f_n|_A)_n$ is a countable dense subset of $\kk_A$
consisting of continuous functions and $0\in \kk_A$ is a non-$G_\delta$
point (and obviously continuous). Consider the sets
\[ \lbf^A=\{ L\in [\nn]: (f_n|_A)_{n\in L} \text{ is pointwise convergent on } A\}\]
and
\[ \mathcal{L}^A_{\mathbf{f},0}=\{ L\in [\nn]: (f_n|_A)_{n\in L} \text{ is
pointwise convergent to } 0 \text{ on } A\}.\]
Let $\phi$ be a $\PB^1_1$-rank on $\lbf^A$ and $\psi$ a $\PB^1_1$-rank on
$\mathcal{L}^A_{\mathbf{f},0}$. We claim that there does not exist a map
$\Phi:\omega_1\to\omega_1$ such that $\psi(L)\leq \Phi\big(\phi(L)\big)$
for all $L\in\mathcal{L}^A_{\mathbf{f},0}$. Assume not. Let
\[ R=\{ L\in [\nn]: \exists \sg\in 2^\nn \text{ with } s_n\sqsubset \sg
\ \forall n\in L\}.\]
Then $R$ is a closed subset of $\lbf^A$. For every $L\in R$, let
$\sg_L=\bigcup_{n\in L} s_n\in 2^\nn$. The map $R\ni L \mapsto
\sg_L\in 2^\nn$ is clearly continuous. Observe that for every $L\in R$
we have that $L\in \mathcal{L}^A_{\mathbf{f},0}$ if and only if
$\sg_L\notin A$. As $R$ is a Borel subset of $\lbf^A$, by boundedness
we get that $\sup\{ \phi(L):L\in R\}=\xi<\omega_1$. Let
$\zeta=\sup\{ \Phi(\lambda): \lambda\leq \xi\}$. The set
$B=R\cap \{L\in \mathcal{L}^A_{\mathbf{f},0}: \psi(L)\leq\zeta\}$ is Borel
and $B=R\cap \mathcal{L}^A_{\mathbf{f},0}$. Hence, the set
$\Sigma_B=\{\sg_L:L\in B\}$ is an analytic subset of $2^\nn\setminus A$.
As $2^\nn\setminus A$ is $\PB^1_1$-true, there exists
$\sg_0\in 2^\nn\setminus A$ with $\sg_0\notin \Sigma_B$.
Pick $L\in R$ with $\sg_L=\sg_0$. Then $L\in B$ yet $\sg_L\notin
\Sigma_B$, a contradiction.
\end{examp}
Although we cannot, in general, bound the order of the tree $S_L$
by that of $T_L$, the following proposition shows that this is
possible for an important special case.
\begin{prop}
\label{newp3} Let $X$ be locally compact, $\kk$ a separable Rosenthal
compact on $X$, $\bs$ a dense sequence in $\kk$ consisting of continuous
functions and $f\in \kk$. If $f$ is continuous, then $o(S_L)\leq o(T_L)$
for all $L\in \llf$.

In particular, there exists a $\PB^1_1$-rank $\phi$ on $\lbf$ and a
$\PB^1_1$-rank $\psi$ on $\llf$ with $\psi(L)\leq\phi(L)$ for all
$L\in\llf$.
\end{prop}
\begin{proof}
We will show that for every $d\in\nn$ we have $o(S^d_L)\leq o(T^d_L)$
for every $L\in\llf$. This clearly completes the proof. So fix $d\in\nn$
and $L\in\llf$. We shall construct a monotone map $\Phi:S^d_L\to
[\fin(L)]^{<\nn}$ such that for every $(s,w)\in S^d_L$ the following hold.
\begin{enumerate}
\item[(i)] $|(s,w)|=|\Phi\big((s,w)\big)|$.
\item[(ii)] If $s=(n_0<...<n_k)$, $w=(l_0,...,l_k)$ and $\Phi\big((s,w)\big)=
(F_0<...<F_k)$, then for every $i\in\{0,...,k\}$ and every $z\in B_{l_i}$
there exists $m_i\in F_i$ with $|f_{n_i}(z)-f_{m_i}(z)|>\frac{1}{d+1}$.
\end{enumerate}
Assuming that $\Phi$ has been constructed, let $M:S^d_L\to T^d_L$ be
defined by
\[ M\big( (s,w)\big)= \big(s,\Phi\big((s,w)\big), w\big). \]
Then it is easy to see that $M$ is a well-defined monotone map, and so,
$o(S^d_L)\leq o(T^d_L)$ as desired.

We proceed to the construction of $\Phi$. It will be constructed by
recursion on the length of $(s,w)$. We set $\Phi\big((\varnothing,
\varnothing)\big)=(\varnothing)$. Let $k\in\nn$ and assume that
$\Phi\big((s,w)\big)$ has been defined for every $(s,w)\in S^d_L$
with $|(s,w)|\leq k$. Let $(s',w')=(s^{\con}n_k, w^{\con}l_k)
\in S^d_L$ with $|s'|=|w'|=k+1$. By the definition of $S^d_L$,
we have that $|f_{n_k}(z)-f(z)|>\frac{1}{d}$ for every
$z\in B_{l_k}$. Put $p=\max\big\{n: n\in F \text{ and }
F\in \Phi\big((s,w)\big)\big\}\in\nn$. For every $z\in B_{l_k}$
we may select $m_z\in L$ with $m_z>p$ and such that
$|f_{n_k}(z)-f_{m_z}(z)|>\frac{1}{d+1}$. As the functions
$(f_n)_n$ are continuous, we pick  an open neighborhood $U_z$ of
$z$ such that $|f_{n_k}(y)-f_{m_z}(y)|>\frac{1}{d+1}$
for all $y\in U_z$. By the compactness of $B_{l_k}$,
there exists $z_0,...,z_{j_k}\in B_{l_k}$ such that
$U_{z_0}\cup ...\cup U_{z_{j_k}} \supseteq B_{l_k}$.
Let $F_k=\{ m_{z_i}:i=0,...,j_k\}\in \fin(L)$ and notice
that $F\leq p<F_k$ for every $F\in \Phi\big((s,w)\big)$. We set
\[ \Phi\big((s',w')\big)=\Phi\big((s,w)\big)^{\con} F_k
\in [\fin(L)]^{<\nn}. \]
It is easy to check that $\Phi\big((s',w')\big)$ satisfies (i)
and (ii) above. The proof is completed.
\end{proof}


\end{document}